\documentclass[12pt]{article}

\usepackage{amssymb}
\usepackage{amsthm}
\usepackage{amsfonts}
\usepackage[latin2]{inputenc}
\usepackage{amsmath}
\usepackage{anysize}
\usepackage{hyperref}
\marginsize{2.6cm}{2.6cm}{1cm}{2cm}

\newtheorem{main_theorem}{Theorem}
\newtheorem{main_definition}{Definition}

\newtheorem{theorem}{Theorem}[section]
\newtheorem{lemma}[theorem]{Lemma}
\newtheorem{corollary}[theorem]{Corollary}
\newtheorem{remark}[theorem]{Remark}

\newtheorem{definition}[theorem]{Definition}
\newtheorem{proposition}[theorem]{Proposition}

\def\o{\omega}

\begin{document}
\title{Weak geodesic rays in the space of K\"ahler potentials and the class $\mathcal E(X,\omega)$}
\author{Tam\'as Darvas\thanks{Research supported by NSF grant DMS1162070 and the Purdue Research Foundation \newline 
2010 Mathematics subject classification 53C55, 32W20, 32U05.}}
\date{\vspace{0.1in} \emph{\small{To Timea.}}\vspace{-0.2in}}
\maketitle
\begin{abstract}
Suppose $(X,\omega)$ is a compact K\"ahler manifold. In the present work we propose a construction for weak geodesic rays in the space of K\"ahler potentials that is tied together with properties of the class $\mathcal{E}(X,\omega)$. As an application of our construction, we prove a characterization of $\mathcal{E}(X,\omega)$ in terms of envelopes.
\end{abstract}
\section{Introduction and main results}

Given $(X^n,\omega)$, a connected compact K\"ahler manifold, the space of smooth K\"ahler potentials is the set
\[\mathcal{H} := \left\{ v \in C^\infty(X) : \omega + i \partial \overline{\partial} v > 0\right\}.\]
This space has a Fr\'{e}chet manifold structure as an open subset of $C^{\infty}(X)$. For $v \in \mathcal{H}$ one can identify $T_v \mathcal{H}$ with $C^{\infty}(X)$. As found by Mabuchi, one can define a Riemannian metric on $T_v \mathcal{H}$:
\[\langle \xi,  \eta \rangle_v := \int_X \xi \eta (\omega + i\partial \overline{\partial}v)^n, \ \ \ \ \xi,\eta \in T_v \mathcal{H}. \]

A smooth curve $(\alpha,\beta)\ni t \to \phi_t \in \mathcal{H}$ with $\alpha,\beta \in \Bbb R \cup \{ -\infty,+\infty\}$ is a geodesic in this space if:
$$\ddot{\phi_t} - \frac{1}{2} \langle \nabla\dot \phi_t, \nabla\dot \phi_t\rangle_{\phi_t}=0, \ t \in (\alpha,\beta).$$

As discovered independently by Semmes \cite{s} and Donaldson \cite{d1}, the above equation can be understood as a complex Monge-Amp\`ere equation. With the notation
$$S_{\alpha\beta} = \left\{s \in \Bbb C: \alpha <\textup{Re }s<\beta \right\}$$ let $\omega$ be the pullback of the K\"{a}hler form $\omega$ to the product ${S_{\alpha\beta}} \times X$. Let $u \in C^{\infty}({S_{\alpha\beta}} \times X)$ be the complexification of $\phi$, defined by $u(s,x) := \phi(\textup{Re}s, x)$. Then $\phi$ is a geodesic if and only if the following equation is satisfied for $u$:
\begin{equation}\label{geodesic_eq}
(\pi^*\omega + i \partial \overline{\partial}u)^{n+1}=0,
\end{equation}
where $\pi: S \times X \to X$ is the projection map to the second component. By analogy with the smooth setting, a curve $ (\alpha,\beta) \ni t \to u_t \in \textup{PSH}(X,\omega)$ is called a weak subgeodesic segment, if its complexification $u: S_{\alpha\beta} \times X \to R$  is a locally bounded $\pi^* \omega$-psh function. If additionally (\ref{geodesic_eq}) is satisfied in the Bedford-Taylor sense \cite{bt}, then $t \to u_t$ is called a weak geodesic segment.

In case $\alpha =0$ and $\beta = \infty$, we call $t\to u_t$ a weak geodesic ray. Given such a curve, we would like to understand the limit $u_\infty := \lim_{t \to \infty} u_t$ whenever it exists. This question is partially motivated by Donaldson's program on the existence and uniqueness of constant scalar curvature K\"ahler metrics in a fixed K\"ahler class. According to this program, one should study the limit behavior of geodesics rays (as well as certain functionals along the rays) as $t \to \infty$. However, for a general weak geodesic ray, $\lim_{t \to \infty} u_t$ does not exist and we need a normalization procedure that fixes this issue. This is the main motivation of our first result, which holds for arbitrary weak geodesic segments not just rays.

If the weak geodesic $t \to u_t$ is in $C^1$, Berndtsson \cite[Section 2.2]{br2} observed that the range of the tangent vectors $\dot u_t : X \to \Bbb R$ is the same for any $t \in (\alpha,\beta)$. If $u$ is a smooth strong geodesic, more can be said, as it is well known that $\dot u_t = \dot u_0\circ F_t$, for some family of symplectomorphisms $F_t:X \to X$  (\cite[Formula (27)]{rz}). Our first result is a generalization of these observations to arbitrary weak geodesic segments.

\begin{main_theorem}[Theorem 3.4] \label{m_norm_thm} Given a weak geodesic $(\alpha,\beta) \ni t  \to u_t \in \textup{PSH}(X,\omega),$ $ (\alpha,\beta \in \Bbb R \cup \{-\infty,+\infty \}),$ for any $a,b,c,d \in (\alpha,\beta)$ one has
\begin{enumerate}
\item[(i)] $\inf_{X}\frac{u_a-u_b}{a-b}=\inf_{ X}\frac{u_c-u_d}{c-d} = :m_u$.
\item[(ii)] $\sup_{X}\frac{u_a-u_b}{a-b}=\sup_{ X}\frac{u_c-u_d}{c-d} = :M_u$,
\end{enumerate}
Hence, $t \to u_t$ is Lipschitz continuous in $t$, with Lipschitz constant $\max\{|M_u|,|m_u|\}$.
\end{main_theorem}

With this result in our pocket we turn to constructing weak geodesic rays. We say that a weak geodesic segment $t \to u_t$ is linear if  $u_t = u_0 + M_u t=u_0 + m_u t$. If additionally $m_u=M_u=0$, then $t \to u_t$ is constant. Clearly, $u_t$ is not linear if and only if $M_u > m_u$. In this case, if $\alpha > - \infty$, then one can always obtain  $M_u = 0$ and $m_u = -1$ by a translation and re-scaling:
$$\tilde u_t = u_{\alpha +\frac{t-\alpha}{M_u - m_u}}(x) - \frac{M_u (t-\alpha)}{M_u - m_u}.$$
One is naturally lead to the following notion, that will be especially useful in our investigation of weak geodesic rays:

\begin{main_definition} We say that a weak geodesic segment $ (\alpha,\beta) \ni t \to u_t \in  \textup{PSH}(X,\omega)$ is normalized if it is constant or $M_u=0$ and $m_u=-1$.
\end{main_definition}

Given a normalized weak geodesic ray $t \to u_t$, by Lipschitz continuity (Theorem \ref{m_norm_thm}) we have that $u_0 = \lim_{t \to 0}u_t \in \textup{PSH}(X,\omega)\cap L^\infty(X)$. Since this last limit is uniform in $X$, it also follows that $\sup_X (u_{t_1} - u_{t_0}) = 0, \ t_1,t_0 \in [0,+\infty), t_0 < t_1$ hence:
$$\begin{cases}
\ u_{t_1} \leq u_{t_0},\ 0 \leq t_0 < t_1,\\
\ \inf_{X}u_0\leq \sup_{X}u_t \leq \sup_{X}u_0, \ t \in (0,+\infty).
\end{cases}$$
It is well known that for any $C \in \Bbb R$ sets of the type
$$\{ v \in \textup{PSH}(X,\omega)| -C \leq \sup v \leq C\}$$
are compact in $\textup{PSH}(X,\omega)$ equipped with the $L^1(X)$ topology.
This implies that the decreasing pointwise limit $u_{\infty}=\lim_{t \to +\infty}u_t$ is $\omega$-psh and it is different from $-\infty$.

For $\phi,\psi \in \textup{PSH}(X,\omega)$, $\psi \leq \phi$ with $\phi$ bounded and $\psi$ possibly unbounded, our goal is to construct a normalized weak geodesic ray $t \to v_t$ such that $v_0 = \phi$ and $v_\infty =\psi$. As we shall see this is not always possible, but whenever it can be done, our construction below will provide such a ray (see Corollary \ref{infinitypotential}).

We introduce the following set of normalized weak geodesic rays
$$\mathcal R(\phi,\psi) = \{ v_t\textup{ is a normalized weak ray with } v_0=\lim_{t \to 0}v_t= \phi \textup{ and } v_\infty=\lim_{t \to \infty}v_t\geq \psi \},$$
where the limits are pointwise, but by Theorem \ref{m_norm_thm} the first is perforce uniform.
By $(0,l) \ni t \to u^l_t \in \textup{PSH}(X,\omega)$ we denote the unique weak geodesic segments joining $\phi$ with $\max\{\phi-l,\psi\}$.

Additionally, let $c_\psi$ be the limit
$$c_\psi = \lim_{l \to +\infty} \frac{AM(\max\{-l,\psi\})}{l},$$
where $AM( \cdot)$ is the Aubin-Mabuchi energy of a bounded $\omega$-psh function. As we shall see in Section 2.3, the constant $c_\psi$ is well defined and finite.

\begin{main_theorem}[Theorem 4.1] \label{intr_ray_const}For any $\phi,\psi \in \textup{PSH}(\omega)$ with $\phi$ bounded and $\psi \leq \phi$, the weak geodesic segments $u^l$ form an increasing family. The upper semicontinuous regularization of their limit $v(\phi,\psi) = \textup{usc}(\lim_{l \to \infty} u^l)$ is a weak geodesic ray for which the following hold:
\renewcommand{\theenumi}{(\roman{enumi})}
\begin{enumerate}
\item[(i)] $v(\phi,\psi)_t = \textup{usc}(\lim_{l \to \infty} u^l_t)$ for any $t \in (0,+\infty)$.

\item[(ii)] $v(\phi,\psi)  \in \mathcal R(\phi,\psi)$, more precisely $v(\phi,\psi) = \inf_{v \in \mathcal R(\phi,\psi)}v$. In particular, $t \to v(\phi,\psi)_t$ is constant if and only if $\mathcal R(\phi,\psi)$ contains only the constant ray $\phi$.

\item[(iii)] $AM(v(\phi,\psi)_t)=AM(\phi)+c_{\psi}t$, in particular $t \to v(\phi,\psi)_t$ is constant if and only if $\psi \in \mathcal E(X,\omega)$.
\end{enumerate}
\end{main_theorem}

In a nutshell, the above theorem says that the ray $v(\phi,\psi)$ is the lower envelope of the elements of $\mathcal R(\phi,\psi)$ and it is constant if and only if $\mathcal R(\phi,\psi)$ contains only the constant ray $t \to \phi$, which in turn is equivalent to $\psi \in \mathcal E(X,\omega)$.

For the definition of the class $\mathcal E(X,\omega) \subset \textup{PSH}(X,\omega)$ we refer the reader to Section 2.3. This class was introduced in \cite{gz} and it was used to solve global Monge-Amp\`ere equations with very rough data. As an intermediate result, we prove in Section 2.3 that $\psi \in \mathcal E(X,\omega)$ if and only if $c_\psi=0$.

Ever since the importance of geodesic rays was pointed out in Donaldson's program \cite{d1}, there have been many papers on methods how to construct weak geodesic rays. We mention \cite{at,ps,c2,ct,rwn1}, to indicate only a few articles in a very fast expanding literature. At first sight, our construction below is perhaps most reminiscent of \cite{c2} (we construct our ray out of segments as well), but our conclusions and the questions investigated seem to be entirely different.

The elements $v \in \mathcal E(X,\omega)$ are usually unbounded but have very mild singularities. In particular, by Corollary 1.8 \cite{gz}, at any $x \in X$ the Lelong number of $v$ is zero. However, as noted in \cite{gz}, this property does not characterize $\mathcal E(X,\omega)$. Our next theorem tries to fill this void, that is to characterize elements of $\mathcal E(X,\omega)$ in terms of the mildness of their singularities.

Given usc functions $b_0,b_1$ one can define the envelopes
$$P(b_0) = \sup \{ \psi \leq b_0: \psi \in \textup{PSH}(X,\omega)\},$$
$$P(b_0,b_1)= P(\min \{ b_0,b_1\}).$$
As the upper semicontinuous regularization $\textup{usc}(P(b_0))$ is a competitor and $P(b_0) \leq \textup{usc}(P(b_0))$, it follows that $P(b_0) \in \textup{PSH}(X,\omega)$ and the same is true for $P(b_0,b_1)$.

For $\psi,\psi' \in \textup{PSH}(X,\omega)$ we say that $\psi$ and $\psi'$ have the same singularity type if and only if there exists $C >0$ s.t.
$$\psi' - C < \psi < \psi' + C.$$
This induces an equivalence relation on $\textup{PSH}(X,\omega)$ and we denote each class by $[\psi]$, given a representative $\psi \in \textup{PSH}(X,\omega)$.

Suppose now that $\phi,\psi \in \textup{PSH}(X,\omega)$ with $\phi \in L^\infty(X)$. In \cite{rwn1} the envelope of $\phi$ with respect to the singularity type of $\psi$ was considered in the following manner:
$$P_{[\psi]}(\phi)= \textup{usc}\Big(\lim _{C \to +\infty}P(\psi+C,\phi)\Big).$$
Assuming that $\psi$ has analytic singularities, one can show that $\psi$ has the same singularity type as $P_{[\psi]}(\phi)$. After making this observation, in \cite[Remark 4.6]{rwn1},\cite[Remark 3.9]{rwn2} the authors ask weather this holds for general $\psi$. This is not the case, as our next result says that given a continuous potential $\phi$, the singularities of the envelope $P_{[\psi]}(\phi)$ disappear once $\psi \in \mathcal E(X,\omega)$.

\begin{main_theorem}[Theorem 5.2] Suppose $\psi \in \textup{PSH}(X,\omega)$ and $\phi \in \textup{PSH}(X,\omega) \cap C(X)$. Then $\psi \in \mathcal E(X, \omega)$ if and only if
$$P_{[\psi]}(\phi)=\phi.$$
\end{main_theorem}
Interestingly, one direction in the proof of this theorem relies on the findings of Theorem 2, although in the statement of this result weak geodesics are not mentioned at all. What seems even more intriguing, the above result can be used to construct very general geodesic segments joining points of $\mathcal E(X, \omega)$ (see \cite{d2}).

Suppose that $\phi,\psi$ are as specified in the above theorem and $\psi$ has a non-zero Lelong number at $x \in X$. This means that in a neighborhood of $x$ we have $$\psi(y) < c \log \|y-x\| + d$$ for some $c,d >0$. Having alternative definitions of the pluricomplex Green function in mind, one observes that this estimate implies that $P_{[\psi]}(\phi)$ has a logarithmic singularity at $x$. Hence, the condition $P_{[\psi]}(\phi)=\phi$ guarantees that all Lelong numbers of $\psi$ are zero, justifying our earlier claim that the above theorem is a characterization of the elements of $\mathcal E(X,\omega)$ in terms of the mildness of their singularities.

Lastly, we compare our construction of weak geodesic rays to the one in \cite{rwn1}. For details on the terminology we refer to Section 2.4. Courtesy of an argument provided to us by Ross and Witt-Nystr\"om, the condition of 'small unbounded locus' can be removed from the definition of a test curve (Theorem \ref{RWN_main}). As a consequence of this and Theorem \ref{m_norm_thm}, we remark that all weak geodesic rays can be constructed using analytic test configurations (Corollary \ref{RWN_main_inverse}).  At last, we conclude that the geodesic rays we constructed in Theorem \ref{intr_ray_const} can be recovered using very specific test curves:

\begin{main_theorem}[Theorem 6.1] Suppose $\phi,\psi \in \textup{PSH}(\omega)$ with $\phi$ bounded and $\psi \leq \phi$. Then the weak geodesic ray $t \to v(\phi,\psi)_t$  is the same as the  ray obtained from a special test curve $\tau \to \gamma^*_\tau$ using the method of \cite{rwn1}.
\end{main_theorem}

\emph{Acknowledgement.} I would like to thank L. Lempert for his guidance and for his patience during the years. I would also like to thank S. Boucksom, M. Jonsson, Y.A. Rubinstein for useful discussions. Special thanks goes to J. Ross and D. Witt-Nystr\"om for allowing me to include their proof of Theorem \ref{RWN_main}. Finally, I would like to thank the anonymous referee for his suggestions that improved the paper greatly.

\section{Preliminaries}

\subsection{Berndtsson's construction}

In this section we recall Berndtsson's construction (\cite[Section 2.1]{br1}) of a weak solution to the Dirichlet problem associated to the geodesic equation in the space of K\"ahler potentials:

\begin{alignat}{2}\label{bvp_Bern}
&u \in \textup{PSH}(S_{0,1}\times X, \omega)\cap L^\infty(S_{0,1}\times X)\nonumber\\
&(\pi^*\omega + i \partial \overline{\partial}u)^{n+1}=0 \nonumber\\
&u(t+ir,x) =u(t,x) \ \forall x \in X, t \in (0,1), r \in \Bbb R \\
&\lim_{t \to 0,1}u_t(x)=u_{0,1}(x),  \forall x \in X\nonumber,
\end{alignat}
where $u_0,u_1 \in \textup{PSH}(X,\omega)\cap L^\infty(X)$ and the limits are uniform in $X$. We introduce the following set of weak subgeodesics:
$$\mathcal S = \{v \textup{ is a weak subgeodesic with }\lim_{t \to 0,1}v_t \leq u_{0,1} \},$$
where the boundary limits are assumed to be only pointwise in $X$. Our candidate solution is defined by taking the upper envelope of this family:
$$u = \sup_{v \in S}v.$$
Since $u_0,u_1$ are bounded, for $A>0$ big enough we have  that $v_t = \max\{u_0 - At, u_1 - A(1-t) \}\in \mathcal S$. Since all elements of $S$ are convex in $t$, we have the estimate:
\begin{equation}\label{Bern_est}
v_t \leq u_t \leq u_0 + t(u_1 - u_0).
\end{equation}
This estimate is also true for the upper semicontinuous regularization $u^*$ of $u$ as well, hence $u^* \in \mathcal S$ implying $u=u^*$. The fact
$$(\pi^*\omega + i \partial \overline{\partial}u)^{n+1}=0$$
follows now from Bedford-Taylor theory adapted to this setting. Uniqueness follows from the maximum principle \cite[Theorem 6.4]{b1}, as we assumed in (\ref{bvp_Bern}) that the boundary limits are uniform. What is more, it follows from (\ref{Bern_est}) that $t \to u_t$ is uniformly Lipschitz continuous in the $t$ variable.

Although we will not use it here, let us mention that for $u_0,u_1 \in \textup{PSH}(X,\omega)\cap C^\infty(X)$ it was proved in \cite{c1} that $u$ has bounded Laplacian, and this regularity is optimal as later observed in \cite{dl}.

\subsection{The Aubin-Mabuchi energy}

The Aubin-Mabuchi energy is a concave functional $AM : \textup{PSH}(X,\omega)\cap L^\infty(X) \to \Bbb R$, given by the formula:
$$AM(v)=\frac{1}{(n+1)}\sum_{j=0}^n\int_{X} v\omega^j\wedge (\omega + i\partial\bar\partial v)^{n-j}.$$
One can easily compute that for $u,v \in \textup{PSH}(X,\omega)\cap L^\infty(X)$ we have
\begin{equation}\label{AM_diff}AM(u)-AM(v)=\frac{1}{(n+1)}\sum_{j=0}^n\int_{X} (u-v)(\omega +i\partial\bar\partial u)^j\wedge (\omega + i\partial\bar\partial v)^{n-j},
\end{equation}
in particular
\begin{equation}\label{AM_const}AM(u + c)=AM(u)+c, \ c \in\Bbb R.
\end{equation}
The Aubin-Mabuchi functional is important for geodesics in the space of K\"ahler potentials because of the following result:

\begin{theorem}\textup{\cite[Proposition 6.2]{bbgz}} \label{AM_geod} For a weak subgeodesic $(\alpha,\beta) \ni t \to u_t \in \textup{PSH}(X,\omega)$,  the correspondence $t \to AM(u_t), \ t \in (\alpha,\beta)$ is convex . Moreover, the subgeodesic $t \to u_t$ is a geodesic  if and only if $t \to AM(u_t)$ is linear.
\end{theorem}

The following well known estimate will be essential in our later investigations about the Aubin-Mabuchi energy. For completeness we include a proof here.

\begin{proposition} \textup{\cite[Proposition 2.8]{begz}}\label{AM_est} For $u \in \textup{PSH}(X,\omega)\cap L^\infty(X)$ with $u \leq 0$ the following holds:
$$\int_X u (\omega+i\partial\bar\partial u)^n \leq AM(u)\leq \frac{1}{n+1}\int_X u (\omega+i\partial\bar\partial u)^n$$
\end{proposition}

\begin{proof} As $u \leq 0$, the second estimate is trivial. To prove the first estimate, we observe that it is enough to argue that
$$\int_X u \omega^k \wedge (\omega + i\partial \bar\partial u)^{n-k} \leq \int_X u \omega^{k+1} \wedge (\omega + i\partial \bar\partial u)^{n-k-1}, \ k \in \{0,\ldots,n-1 \}.$$
However this follows easily as we have
$$\int_X u \omega^k \wedge (\omega + i\partial \bar\partial u)^{n-k} = \int_X u \omega^{k+1} \wedge (\omega + i\partial \bar\partial u)^{n-k-1}-\int_X i\partial u \wedge \bar \partial u \wedge \omega^{k} \wedge (\omega + i\partial \bar\partial u)^{n-k-1}.$$
\end{proof}

The last result in this section is a kind of domination principle for the Aubin-Mabuchi energy. Although we could not find a reference for it, its proof is implicit in many standard arguments throughout the literature. See for example \cite[Theorem 1.1]{b2}.
\begin{proposition} \label{energy_dom}Suppose $u,v \in \textup{PSH}(X,\omega)\cap L^\infty(X)$ with $u \geq v$. If $AM(u)=AM(v)$, then $u=v$.
\end{proposition}

\begin{proof}
Since $u \geq v$ it follows from (\ref{AM_diff}) that
\begin{equation}\label{egy}
\int_X (u - v)(\omega +i\partial\bar\partial u)^j\wedge (\omega + i\partial\bar\partial v)^{n-j}=0, \ j = 0,...,n.
\end{equation}
We are finished if we can prove that
\begin{equation}\label{main_id}
\int_X i\partial(u - v)\wedge \bar\partial(u - v)\wedge \omega^{n-1}=0.
\end{equation}
The first step is to prove that
\begin{equation}\label{main_id1}
\int_X i \partial(u - v)\wedge\bar\partial(u - v)\wedge \omega\wedge(\omega +i\partial\bar\partial u)^j\wedge (\omega + i\partial\bar\partial v)^{n-j-2}=0, \ j = 0,...,n-2.
\end{equation}
It follows from (\ref{egy}) that
\begin{flalign}\label{CauchySchwarz}
0=&\int_X (u - v)i\partial\bar\partial(u - v)\wedge(\omega +i\partial\bar\partial u)^j\wedge (\omega + i\partial\bar\partial v)^{n-j-1} \nonumber\\
 =&-\int_X i \partial(u - v)\wedge \bar\partial(u - v)\wedge(\omega +i\partial\bar\partial u)^j\wedge (\omega + i\partial\bar\partial v)^{n-j-1}, \ j = 0,...,n-1.
\end{flalign}
Using this we now write
\begin{flalign*}
&\int_X i \partial(u - v)\wedge \bar\partial(u - v)\wedge\omega\wedge(\omega +i\partial\bar\partial u)^j\wedge (\omega + i\partial\bar\partial v)^{n-j-2}=\\
&=-\int_X i \partial(u - v)\wedge \bar\partial(u - v)\wedge i\partial \bar \partial v\wedge(\omega +i\partial\bar\partial u)^j\wedge (\omega + i\partial\bar\partial v)^{n-j-2}\\
&=\int_X i \partial(u - v)\wedge i\partial\bar\partial(u - v)\wedge \bar \partial v\wedge(\omega +i\partial\bar\partial u)^j\wedge (\omega + i\partial\bar\partial v)^{n-j-2}\\
&=\int_X i \partial(u - v)\wedge \bar \partial v\wedge (\omega + i\partial\bar\partial u)\wedge(\omega +i\partial\bar\partial u)^j\wedge (\omega + i\partial\bar\partial v)^{n-j-2}-\\
&-\int_X i \partial(u - v)\wedge \bar \partial v\wedge (\omega + i\partial\bar\partial v)\wedge(\omega +i\partial\bar\partial u)^j\wedge (\omega + i\partial\bar\partial v)^{n-j-2}, \ j = 0,...,n-2.
\end{flalign*}
Using the Cauchy-Schwarz inequality for the Monge-Amp\`ere operator, it follows from (\ref{CauchySchwarz}) that both of the terms in the last sum are zero, proving (\ref{main_id1}). Continuing this inductive process we arrive at (\ref{main_id}).
\end{proof}

\subsection{The class $\mathcal{E}(X,\omega)$}

We recall here a few facts about the class $\mathcal E(X,\omega) \subset \textup{PSH}(X,\omega)$. For the sake of brevity, we take a very minimalistic approach in our presentation. For a more complete treatment we refer the reader to \cite{gz}. For $\gamma \in \textup{PSH}(X,\omega)$, one can define the canonical cutoffs $\gamma_l \in \textup{PSH}(X,\omega), \ l \in \Bbb R$ by the formula
$$\gamma_l = \max \{-l, \gamma \}.$$
By an application of the comparison principle, it follows that the Borel measures
$$\chi_{\{\gamma > -l\}}(\omega + i\partial \bar\partial \gamma_l)^n$$
are increasing in $l$.
Following \cite{begz}, despite the fact that $\gamma$ might be unbounded, one can still make sense of $(\omega + i\partial\bar\partial \gamma)^n$ as the limit of these increasing measures:
\begin{equation}\label{non-plurip}(\omega + i\partial\bar\partial \gamma)^n= \lim_{l \to +\infty} \chi_{\{\gamma > -l\}}(\omega + i\partial \bar\partial \gamma_l)^n.
\end{equation}
Using this definition, $(\omega + i\partial\bar\partial \gamma)^n$ is called the non-pluripolar Monge-Amp\`ere measure of $\gamma$. It is clear from (\ref{non-plurip}) that
$$\int_X (\omega + i\partial\bar\partial \gamma)^n \leq \int_{X}\omega^n =\text{Vol}(X).$$
This brings us to the class $\mathcal E(X,\omega)$. By definition, $\gamma \in \mathcal E(X,\omega)$ if
\begin{equation}\label{Eps_def}
\int_X (\omega + i\partial\bar\partial \gamma)^n=\lim_{l \to \infty} \int_X \chi_{\{\gamma > -l\}}(\omega + i\partial \bar\partial \gamma_l)^n =\text{Vol}(X).
\end{equation}

As detailed in \cite{gz}, most of the classical theorems of Bedford-Taylor theory are valid for the class $\mathcal E(X,\omega)$ as well. For brevity we only mention here a version of the domination principle that we will need later.

\begin{proposition}\textup{\cite{bs}} \label{domination}Suppose $\psi \in \mathcal E(X,\omega)$ and $\phi \in \textup{PSH}(X,\omega) \cap L^\infty(X)$. If $\psi \geq \phi$ a.e. with respect to $(\omega + i\partial\bar\partial \psi)^n$ then $\psi(x) \geq \phi(x), \ x \in X$.
\end{proposition}

\begin{proof} We follow the proof in \cite{bs}, which in turn is based on and idea of Zeriahi. By translation, we can assume that both $\phi,\psi$ are negative on $X$. Then
for all $s > 0$ and small enough $\varepsilon >0$ we have:
$$\{\psi - \phi < -s - \varepsilon\phi\}\subseteq\{\psi - \phi < 0\}.$$
Now we can use Lemma 2.3 in \cite{egz}, that is easily seen to be valid for elements of $\mathcal E(X,\omega)$ (one only needs the comparison principle, which is true for elements of $\mathcal E(X,\omega)$, as is proved in \cite{gz}). According to this result we have:
$$\varepsilon^n \textup{Cap}_{\omega}(\{ \psi - \phi < -s - \varepsilon\})\leq \int_{\{\psi - \phi < -s -\varepsilon \phi\}}(\omega + i\partial\bar\partial \psi)^n \leq \int_{\{\psi - \phi < 0\}}(\omega + i\partial\bar\partial \psi)^n =0 .$$
Since the Monge-Amp\`ere capacity dominates the Lebesgue measure, it results that $\psi - \phi \geq -s - \varepsilon$ a.e. with respect to $\omega^n$. By the sub mean-value property of plurisubharmonic functions this estimate extends to $X$. Now letting $s,\varepsilon \to 0$ we obtain the desired result.
\end{proof}

Next, we observe that $l \to \gamma_l = \max\{-l,\gamma\}, \ l \geq 0$ is a decreasing weak subgeodesic ray, hence the map $l \to AM(\gamma_l), \ l \in (0,+\infty)$ is convex and decreasing by Theorem \ref{AM_geod} and (\ref{AM_diff}). From this, it follows that the following quantity is well defined and finite:
$$c_\gamma = \lim_{l \to +\infty} \frac{AM(\gamma_l)}{l} = \lim_{l \to +\infty} \frac{AM(\gamma_l) - AM(\gamma_0)}{l}.$$
Our next result gives a precise formula for $c_\gamma$, one that does not use subgeodesic rays. We also obtain a characterization of  $\mathcal E(X,\omega)$ in terms of this constant.
\begin{theorem} \label{E_energy}Given $\gamma \in \textup{PSH}(X,\omega)$ we have that
$$c_\gamma = \frac{-1}{n+1}\sum_{j=0}^n \lim_{l \to +\infty}\int_{\{\gamma \leq -l\}} \omega^{j} \wedge(\omega +i\partial\bar\partial \gamma_l)^{n-j}.$$
In particular, $\gamma \in \mathcal E(X,\omega)$ if and only if $c_\gamma =0$.
\end{theorem}
\begin{proof} It follows from (\ref{AM_const}) that for any $d  \in \Bbb R$ we have $c_{\gamma +d}=c_\gamma$.  Hence, we can assume that $\gamma < 0$.
First we prove that
\begin{equation}\label{main_lim}
\lim_{l \to +\infty}\int_X \frac{\gamma_l}{l} (\omega+i\partial\bar\partial \gamma_l)^n = -\lim_{l \to +\infty}\int_{\{\gamma \leq -l\}} (\omega+i\partial\bar\partial \gamma_l)^n.
\end{equation}
To see this, since
$$\int_X \frac{\gamma_l}{l} (\omega+i\partial\bar\partial \gamma_l)^n =-\int_{\{\gamma \leq -l\}} (\omega+i\partial\bar\partial \gamma_l)^n+\int_{\{\gamma > -l\}}\frac{\gamma}{l}(\omega+i\partial\bar\partial \gamma_l)^n,$$
it is enough to show that
\begin{equation}
\lim_{l \to +\infty}\int_{\{\gamma > -l\}}\frac{\gamma}{l}(\omega+i\partial\bar\partial \gamma_l)^n = 0.
\end{equation}
Since $\gamma <0$, it is enough to prove that
\begin{equation}\label{lim_eq}
\lim_{l \to +\infty}\int_{\{\gamma > -l\}}\frac{\gamma}{l}(\omega+i\partial\bar\partial \gamma_l)^n \geq 0.
\end{equation}
For any $\varepsilon >0$, we have
$$\int_{\{\gamma > -l\}}\frac{\gamma}{l}(\omega+i\partial\bar\partial \gamma_l)^n\geq  -\int_{\{-\varepsilon l \geq \gamma > -l\}}(\omega+i\partial\bar\partial \gamma_l)^n -\varepsilon  \int_{\{ \gamma > -\varepsilon l \}}(\omega+i\partial\bar\partial \gamma_l)^n. $$
The second term in the sum is bounded below by $-\varepsilon \text{Vol}(X)$, whereas the the first term can be written:
\begin{flalign*}
\int_{\{-\varepsilon l \geq\gamma > -l\}}(\omega+i\partial\bar\partial \gamma_l)^n &= -\int_{\{\gamma > -\varepsilon l\}}(\omega+i\partial\bar\partial \gamma_l)^n + \int_{\{\gamma > -l\}}(\omega+i\partial\bar\partial \gamma_l)^n\\
&= -\int_{\{\gamma > -\varepsilon l\}}(\omega+i\partial\bar\partial \gamma_{\varepsilon l})^n + \int_{\{\gamma > -l\}}(\omega+i\partial\bar\partial \gamma_l)^n,
\end{flalign*}
where the last line follows from locality of the Monge-Amp\`ere measure in the plurifine topology (see formula (2) in \cite{gz}). Since the measures $\chi_{\{\gamma > -l\}}(\omega + i\partial \bar\partial \gamma_l)^n$ increase with $l$, taking the limit in the above identity we obtain
$$\lim_{l \to \infty}\int_{\{-\varepsilon l \geq\gamma > -l\}}(\omega+i\partial\bar\partial \gamma_l)^n=0,$$
proving (\ref{lim_eq}), which in turn implies (\ref{main_lim}). Following exactly the same line of thought, one can prove that
\begin{equation}
\lim_{l \to +\infty}\int_X \frac{\gamma_l}{l} \omega^j \wedge (\omega+i\partial\bar\partial \gamma_l)^{n-j} = -\lim_{l \to +\infty}\int_{\{\gamma \leq -l\}} \omega^j \wedge (\omega+i\partial\bar\partial \gamma_l)^{n-j}, \ j \in {0,\ldots,n}.
\end{equation}
After we sum over $j$ in the above equation, we obtain the formula for $c_\gamma$ in the statement of the theorem.

To prove the last claim, observe that from Proposition \ref{AM_est} it follows that
\begin{equation}\label{est_line}
\int_X \frac{\gamma_l}{l} (\omega+i\partial\bar\partial \gamma_l)^n \leq \frac{AM(\gamma_l)}{l}\leq \frac{1}{n+1}\int_X \frac{\gamma_l}{l} (\omega+i\partial\bar\partial \gamma_l)^n, \ l >0.
\end{equation}
Putting (\ref{est_line}) and (\ref{main_lim}) together we obtain
$$-\lim_{l \to +\infty}\int_{\{\gamma \leq -l\}} (\omega+i\partial\bar\partial \gamma_l)^n \leq \lim_{l \to +\infty}\frac{AM(\gamma_l)}{l}=c_\gamma\leq \frac{-1}{n+1}\lim_{l \to +\infty}\int_{\{\gamma \leq -l\}} (\omega+i\partial\bar\partial \gamma_l)^n.$$
By the definition of the class $\mathcal E(X,\omega)$, $\lim_{l \to +\infty}\int_{\{\gamma \leq -l\}} (\omega+i\partial\bar\partial \gamma_l)^n=0$ if and only if $\gamma \in \mathcal E(X,\omega)$, proving the result.
\end{proof}
\begin{remark} \label{c_alt_def} In the definition of $c_\gamma$ we could have started with the more general decreasing weak subgeodesic ray
$$l \to \tilde \gamma_l = \max \{ \beta - l,\gamma \}$$
for any $\beta \in L^\infty(X) \cap \textup{PSH}(X,\omega)$. As it is easily verified, the resulting constant $\tilde c_\gamma = \lim_{l \to + \infty }AM(\tilde \gamma_l)/l$ is the same as our original $c_\gamma$.
\end{remark}
\subsection{The weak geodesic rays of Ross and Witt-Nystr\"om}

In this section we review the construction of Ross and Witt-Nystr\"om \cite{rwn1} of weak geodesic rays. Although \cite{rwn1} is written in the setting when the K\"ahler structure is integral $([\omega] \in H^2(X,\Bbb Z))$, the whole construction  carries over without changes to our more general situation.


\begin{definition} A map $\Bbb R \ni \tau  \to \psi_\tau \in \textup{PSH}(X, \omega)$ is called a test curve if
\renewcommand{\theenumi}{(\roman{enumi})}
\begin{enumerate}\label{testcurvedef}
\item[(i)] $\tau \to \psi_\tau(x)$ is concave in $\tau$ for any $x \in X$,
\item[(ii)] There exists $C_\psi >0$ such that $\psi_\tau$ is equal to some bounded potential $\psi_{-\infty} \in \textup{PSH}(X,\omega)\cap L^\infty(X)$ for $\tau < -C_\psi$ and $\psi_\tau = -\infty$ if $\tau > C_\psi$.
\end{enumerate}
\end{definition}

We remark that in definition of test curves from \cite{rwn1}, it is also assumed that each $\psi_\tau$ has small unbounded locus. As it will be clear in Theorem \ref{rwn_lemma1} below, this condition can be omitted.

Given a test curve $\tau \to \psi_\tau$, the curve of singularity types $\tau \to [\psi_\tau]$ is called an \emph{analytic test configuration}. As mentioned in the introduction, given $b_0,b_1$ usc functions on $X$, one can define the envelopes
$$P(b_0) = \sup \{ \psi \leq b_0: \psi \in \textup{PSH}(\omega)\},$$
$$P(b_0,b_1)= P(\min \{ b_0,b_1\}),$$
that are elements of $\textup{PSH}(X,\omega)$. Given $\phi,\psi \in \textup{PSH}(X,\omega)$, $\phi \in L^\infty(X)$, one can introduce the envelope of $\phi$ with respect to the singularity type of $\psi$. This is given by the formula:
$$P_{[\psi]}(\phi)= \text{usc}\Big(\lim _{D \to +\infty}P(\psi+D,\phi)\Big).$$

In the process of obtaining a weak geodesic ray one starts out with a test curve $\tau \to \psi_\tau$, as defined above, and a potential $\phi \in \textup{PSH}(X,\omega) \cap L^\infty(X)$. The first step is to 'maximize' the test curve $\tau \to \psi_\tau$ with respect to $\phi$ by introducing the new test curve
$$\tilde \psi_\tau = P_{[\psi_\tau]}(\phi).$$
The second step is to take the 'inverse' Legendre transform of $\tilde \psi_\tau$:
\begin{equation}\label{BWNgeodDef}
\phi_t = \text{usc}\Big(\sup_{\tau \in \Bbb R}(\tilde \psi_\tau + t\tau)\Big) \in \textup{PSH}(X,\o) \cap L^\infty(X) , \ t \in [0,+\infty).
\end{equation}

Before we state the main result of \cite{rwn1}, we remark that in \eqref{BWNgeodDef} it is possible to omit the upper semi continuous regularization. In fact, this is more generally true for any test curve:

\begin{proposition} \label{uscomit} Suppose $\tau \to \psi_\tau$ is a test curve as given in Definition \ref{testcurvedef}. Then for any $t \geq 0$ we have:
$$\textup{usc}\big[\sup_{\tau} (\psi_\tau + t\tau)\big] = \sup_{\tau} (\psi_\tau + t\tau).$$ 
\end{proposition}
\begin{proof} We denote $u(s,z) = \sup_{\tau} (\psi_\tau(z) +  \textup{Re }s \tau)$ for $(s,z) \in S \times X$ where $S = \{ \textup{Re }s \geq 0\} \subset \Bbb C$. Clearly, $\textup{usc} \ u \in \textup{PSH}(S\times X,\pi^*\o)$. It will be enough to show that  $\textup{usc}\ u=u$. 

We introduce $E = \{ u < \textup{usc} \ u\} \subset S \times X$. As both $u$ and $\textup{usc }u$ are $\Bbb R$-invariant, it follows that $E$ is also $\Bbb R$-invariant, i.e. there exists $B \subset [0,\infty) \times X$ such that 
$$E = B +i\Bbb R.$$
As $E$ has capacity $0$, it follows that $E$ has Lebesgue measure zero. For $z \in X$, we introduce the slices:
$$B_z = B \cap [0,\infty) \times \{z \}.$$
We have that $B_z$ has Lebesgue measure $0$ for all $z \in X \setminus F$, where $F \subset X$ is some set of Lebesgue measure $0$. 

Let $z \in X \setminus F$. We argue that $B_z$ is in fact empty. Both maps $t \to u(t,z)$ and $t \to (\textup{usc } u)(t,z)$ are convex on $[0,\infty)$. As they agree on the dense set $[0,\infty) \setminus B_z$, it follows that they have to be the same, hence $B_z=0$.

For fixed $\tau \in \Bbb R$ we clearly have
$$\psi_\tau = \inf_{ t\geq 0} [u_t - \tau t] \leq \chi_\tau :=\inf_{ t\geq 0} [{(\textup{usc } u)}_t - \tau t].$$
Because each $B_z$ is empty for $z \in X \setminus F$, it follows that $\psi_\tau = \chi_\tau$ outside the set of measure zero $F$. Since both $\psi_\tau$ and $\chi_\tau$ are $\omega$-psh (the former by definition, the latter by Kiselman's minimum principle) it follows that $\psi_\tau=\chi_\tau$. Applying the Legendre transform to the curves $\tau \to \psi_\tau$ and $\tau \to \chi_\tau$, we obtain that $u = \textup{usc }u$.
\end{proof}

Finally, we state one of the main results of \cite{rwn1}:
\begin{theorem}\textup{\cite[Theorem 1.1]{rwn1}} \label{RWN_main} The curve $[0,+\infty) \ni t \to \phi_t \in \textup{PSH}(X,\omega)$ introduced in \eqref{BWNgeodDef} is a weak geodesic ray emanating from $\phi$.
 \end{theorem}

As mentioned in the beginning of the section, the original definition of a an analytic test  configuration $\tau \to [\psi_\tau]$ also assumed that each $\psi_\tau$ has small unbounded locus in $X$. This condition is superfluous and the above theorem holds in this greater generality. At the recommendation of the referee, we give here a proof provided by Ross and Witt-Nystr\"om:

\begin{proof}[Proof of Theorem \ref{RWN_main}] Let $\tau \to \psi_\tau$ be a test curve as given in Definition \ref{testcurvedef}. Let $h_t$ be the Legendre transform of $\psi_\tau$: 
$$h_t:=\sup_\tau (\psi_\tau + t\tau), \ \  t \geq 0$$ 
i.e. $t \to h_t$ is a bounded subgeodesic ray with $h_0 = \psi_{-\infty}$. 
For $D>0$, let $\phi^D_t$ denote the supremum of all weak subgeodesic rays bounded from above by $\min( \phi+C_\psi t , h_t+D)$. Clearly, $ t \to \phi_t^D$ is also a bounded subgeodesic ray and we introduce the Legendre transform of $\phi^D_t$:
$$\hat \phi^D_\tau =\inf_{t \geq 0} (\phi^D_t - t\tau).$$
As $\phi^D_t \leq \min( \phi+C_\psi t , h_t+D)$, we obtain that in fact $ \hat \phi^D_\tau \leq \min(\psi_\tau + D, \phi)$ for any $\tau \leq C_\psi$ and $\hat \phi^D_\tau=-\infty$ for $\tau >C_\psi$. Also,  by the Kiselman minimum principle $\inf_t(\phi^D_t-t
\tau)$ is $\omega$-psh, hence
$$\hat \phi^D_\tau = \inf_t(\phi^D_t - t\tau) \leq P(\psi_\tau + D, \phi).$$ 
Applying the inverse Legendre transform to the above inequality we obtain that 
\begin{equation}\label{phiDformula}
\phi^D_t=\sup_\tau (P(\psi_\tau + D,\phi)+ t \lambda),
\end{equation}
since $\tau \to P(\psi_\tau +D,\phi)$ is a test curve and the right hand side is subgeodesic ray that is a candidate in the definition of $\phi^D$. 
Let $\phi_t$ denote the limit of $\phi^D_t$: 
$$\phi_t =\textup{usc}\big( \lim_{D \to +\infty}\phi^D_t \big) \in \textup{PSH}(X, \omega).$$
As we will see by the end of the proof, $t \to \phi_t$ introduced this way is the same curve as the one in \eqref{BWNgeodDef}. Using the comparison principle for geodesics, one can show that $t \to \phi_t$ is a weak geodesic ray. Taking the limit in \eqref{phiDformula} we find
$$\lim_{D \to +\infty}\phi^D_t=\sup_\tau (\lim_{D \to +\infty}P(\psi_\tau + D,\phi)+ t \tau).$$

As $P(\psi_\lambda + C,\phi) \leq P_{[\psi_\lambda]}(\phi)$ we have  $\phi^D_t \leq \sup_\tau(P_{[\psi_\lambda]}(\phi)+t \lambda)$. Since $\phi^D_t$ inceases a.e. to $\phi_t$ we get that 
\begin{equation}\label{phitineq1}
\phi_t\leq \sup_\tau(P_{[\psi_\lambda]}(\phi)+t \lambda).
\end{equation}
To argue the other direction, note that since $\phi^D_t \leq \phi_t$ and \eqref{phiDformula} we have
$$\inf_{t \geq 0} (\phi_t - t\tau) \geq P(\psi_\tau +D,\phi).$$
The left hand side is $\omega$-psh by Kiselman's minimum principle and since $P(\psi_\tau+D,\phi)$ increases a.e. to $P_{[\psi_\tau]}(\phi)$, we obtain that $\inf_t (\phi_t - t\tau) \geq P_{[\psi_\tau]}(\phi).$ Taking the Legendre transform of this we obtain that
\begin{equation}\label{phitineq2}
\phi_t\geq \sup_\tau(P_{[\psi_\tau]}(\phi)+t \tau).
\end{equation}
The inequalities \eqref{phitineq1} and \eqref{phitineq2} together imply that $t \to \sup_\tau(P_{[\psi_\tau]}(\phi)+t \tau)$ is indeed a geodesic ray, finishing the proof.
\end{proof}

Lastly, we recall another proposition from \cite{rwn1}, one which will be very useful for us later:

\begin{proposition}\textup{\cite[Theorem 4.10]{rwn1}} Suppose $\phi, \psi \in \textup{PSH}(X,\omega)$ \label{rwn_lemma1}with $\phi$ continuous and $\psi$ possibly unbounded. Then $P_{[\psi]}(\phi)$ is maximal with respect to $\phi$, i.e. $P_{[\psi]}(\phi) = \phi$ a.e. with respect to the measure $(\omega + i \partial\bar \partial P_{[\psi]}(\phi))^n$, where $(\omega + i \partial\bar \partial P_{[\psi]}(\phi))^n$ is defined as in \eqref{non-plurip}.
\end{proposition}

\section{Normalization of weak geodesics}

Given a weak subgeodesic segment $(\alpha,\beta)\ni t \to u_t \in \textup{PSH}(X,\omega)$, since the correspondence $t \to u_t(x)$ is convex for any $x \in X$, the left and right $t$-derivatives $\dot u^+_t(x)$ and $\dot u^-_t(x)$ exist at each point $t \in (\alpha,\beta).$ One can also define $\dot u^+_\alpha= \lim_{t \searrow \alpha}u^+_t$ and $\dot u^-_\beta = \lim_{t \nearrow \beta}u^-_t$, with values possibly equal to $\pm\infty$. Our first lemma is a precise statement about these functions. Given $u_1,u_2 \in \textup{PSH}(X,\omega) \cap L^\infty(X)$, we denote by $u(u_0,u_1)$ the weak geodesic segment joining $u_0$ with $u_1$ constructed by the method of Section 2.1.
\begin{lemma} \label{lemma1}Suppose $u_0,u_1 \in \textup{PSH}(X,\omega) \cap L^\infty(X)$. For $v = u(u_0,u_1)$ we have:
$$\inf_{X}{\dot v}^+_0 = \inf_{X} (u_1 - u_0).$$
$$\sup_{X}{\dot v}^-_1 = \sup_{X} (u_1 - u_0).$$
\end{lemma}

\begin{proof} We only prove the first identity, as the second follows trivially from $u(u_0,u_1)_t = u(u_1,u_0)_{1-t}$.

It follows from the construction in Section 2.1 that $v_t \geq u_0 + \inf_{x \in X} (u_1 - u_0) t, \ t \in[0,1]$. This implies that $\inf_{x \in X}{\dot v}^+_0 \geq \inf_{x \in X} (u_1 - u_0).$ The other direction is easily seen using the uniform Lipschitz continuity and convexity in the $t$ variable:
$$u_1(x) - u_0(x)= \int_0^1 \dot v_t(x)dt \geq \dot v^+_0(x), \ x\in X.$$
\end{proof}
For a weak geodesic segment $(\alpha,\beta) \ni t \to u_t \in \textup{PSH}(X,\omega)$ and $ a,b \in (\alpha,\beta)$, we denote by $(0,1) \ni t \to u^{ab}_t \in \textup{PSH}(X,\omega)$ the rescaled weak geodesic segment $u^{ab}_t=u_{a + (b-a)t}$.

\begin{lemma} \label{lemma2}Given a weak geodesic segment $(\alpha,\beta) \ni t \to u_t \in \textup{PSH}(X,\omega)$, if $\alpha < a<b <\beta$ we have that $u^{ab}=u(u_a,u_b)$.
\end{lemma}

\begin{proof} Choose $c,d$ such that $\alpha < c <a < b < d <\beta$. By convexity, for any $t_1,t_2 \in (a,b)$ we have:
$$\frac{u^{ab}_c-u^{ab}_a}{c-a} \leq  \frac{u^{ab}_{t_1} - u^{ab}_{t_2}}{t_1 - t_2}\leq\frac{u^{ab}_b-u^{ab}_d}{b-d},$$
hence $t \to u^{ab}_t$ is uniformly Lipschitz continuous in the $t$ variable. As it follows from the construction in Section 2.1, this is also true for $u(u_a,u_b)$ as well. Hence, the classical maximum principle can be used to conclude the lemma.
\end{proof}

\begin{lemma} \label{lemma3}Given a weak geodesic segment $ (\alpha,\beta) \ni t \to u_t \in \textup{PSH}(X,\omega)$, for any $\alpha < a<b<c< \beta$ one has:
\renewcommand{\theenumi}{(\roman{enumi})}
\begin{enumerate}
\item[(i)] $\frac{b-a}{c-a}\inf_{X} \frac{u_b - u_a}{b-a} + \frac{c-b}{c-a}\inf_{X} \frac{u_c - u_b}{c-b} \leq \inf_{X} \frac{u_c - u_a}{c-a}$,
\item[(ii)] $\frac{b-a}{c-a}\sup_{X} \frac{u_b - u_a}{b-a} + \frac{c-b}{c-a}\sup_{X} \frac{u_c - u_b}{c-b} \geq \sup_{X} \frac{u_c - u_a}{c-a}$,
\item[(iii)] $\inf_{X} \frac{u_c - u_b}{c-b} \geq \inf_{X} \frac{u_c - u_a}{c-a}$,
\item[(iv)] $\sup_{X} \frac{u_b - u_a}{b-a} \leq \sup_{X} \frac{u_c - u_a}{c-a}$.
\end{enumerate}
\end{lemma}
\begin{proof} The first two estimates are trivial, whereas the last two follow from the convexity of the maps $t \to u_t(x)$.
\end{proof}
\begin{theorem}  \label{norm_thm} Given a weak geodesic $(\alpha,\beta) \ni t  \to u_t \in \textup{PSH}(X,\omega),$ $ (\alpha,\beta \in \Bbb R \cup \{-\infty,+\infty \}),$ for any $a,b,c,d \in (\alpha,\beta)$ one has
\begin{enumerate}
\item[(i)] $\inf_{X}\frac{u_a-u_b}{a-b}=\inf_{ X}\frac{u_c-u_d}{c-d} = m_u$.
\item[(ii)] $\sup_{X}\frac{u_a-u_b}{a-b}=\sup_{ X}\frac{u_c-u_d}{c-d} = M_u$,
\end{enumerate}
Hence, $t \to u_t$ is Lipschitz continuous in $t$, with Lipschitz constant $\max\{|M_u|,|m_u|\}$.
\end{theorem}
\begin{proof} We only prove (i), as the proof of (ii) is similar. We can suppose that $a < b < c < d$. It follows from Lemma \ref{lemma1} and \ref{lemma2} that
\begin{equation}\label{inf1}\inf_{X} \dot u^+_a=\inf_{X} \frac{u_b - u_a}{b-a} = \inf_{X} \frac{u_c - u_a}{c-a}.
\end{equation}
Putting this into Lemma \ref{lemma3} (i) we obtain
$\inf_{x \in X} \frac{u_c - u_b}{c-b} \leq \inf_{x \in X} \frac{u_c - u_a}{c-a}$. Now, Lemma \ref{lemma3} (iii) implies that
\begin{equation}
\label{inf2}\inf_{X} \frac{u_c - u_b}{c-b} = \inf_{X} \frac{u_c - u_a}{c-a}.
\end{equation}
By (\ref{inf1}) and (\ref{inf2}) we obtain that
$$\inf_{X} \frac{u_b - u_a}{b-a} = \inf_{X} \frac{u_c - u_b}{c-b}.
$$
Changing the letters $(a,b,c)$ to $(b,c,d)$ in the above identity we obtain $\inf_{ X} \frac{u_c - u_b}{c-b} = $ $ \inf_{ X} \frac{u_d - u_c}{d-c}$, proving part (i) of the theorem.
\end{proof}

We make now a slight digression. For a weak geodesic segment, the correspondence $t \to u_t(x), \ x \in X$ is convex, hence the pointwise boundary limits $u_{\alpha}(x) = \lim_{t \to \alpha}u_t(x)$ and $u_{\beta}(x) = \lim_{t \to \beta}u_t(x)$ exist even if they might not be bounded or $\omega$-psh. As a corollary to our previous theorem, we obtain that if $\alpha$ or $\beta$ is finite, the boundary potentials $u_{\alpha}$ and $u_{\beta}$ are bounded $\omega$-psh functions and the corresponding limits $\lim_{t\searrow \alpha} u_t =u_\alpha$ and $ \lim_{t\nearrow \beta} u_t =u_\beta$ are uniform in $X$. In essence this proves the following existence and uniqueness theorem, which is only a slight generalization of a result of Berndtsson \cite{br1}.
\begin{corollary} Given bounded $u_0,u_1 \in \textup{PSH}(X,\omega)$, there exists a unique solution to the following Dirichlet problem for locally bounded $\omega-$psh functions $u : S_{01}\times X \to \Bbb R$:
\begin{alignat}{2}\label{bvp}
  &(\pi^* \omega + i \partial \overline{\partial}u)^{n+1}=0 \nonumber\\
  &u(t+ir,x) =u(t,x) \  x \in X, t \in (0,1), r \in \Bbb R \\
  &\lim_{t \to 0}u(t,x)=u_{0}(x) \textup{ and }\lim_{t \to 1}u(t,x)=u_{1}(x),  \ x \in X\nonumber,
\end{alignat}
where the boundary limits are assumed to be only pointwise in $X$.
\end{corollary}

\begin{proof} By the method of Section 2.1, there exists a solution $u$ that is bounded and assumes the boundary values uniformly, not just pointwise. By our discussion above, any other locally bounded solution $v$ is infact globally bounded and assumes the boundary values uniformly as well. Hence by an application of the standard maximum principle adapted to this setting \cite[Theorem 6.4]{b1}, the solution provided by Berndtsson's method is unique.
\end{proof}

We note here that the uniqueness part of the above theorem does not seem to follow from an application of the classical maximum principle alone (without Theorem \ref{norm_thm}), as this result does not work with pointwise boundary limits. Our observation, that in this problem, pointwise boundary limits are infact uniform  seems to be essential.

\section{A construction of weak geodesic rays}

Before we prove our main theorem about constructing weak geodesic rays, let us recall some notation introduced in the beginning. For $\phi,\psi \in \textup{PSH}(X,\omega)$, $\psi \leq \phi$ with $\phi$ bounded and $\psi$ possibly unbounded we define the following set of weak geodesic rays:
$$\mathcal R(\phi,\psi) = \{ v \textup{ is a normalized weak geodesic ray with } \lim_{t \to 0}v_t= \phi \textup{ and } \lim_{t \to \infty}v_t\geq \psi \}.$$

The set $\mathcal R(\phi,\psi)$ is always non-empty as it contains the constant ray $u = \phi$. By $(0,l) \ni t  \to u^l_t \in  \textup{PSH}(X,\omega)$ we denote the unique weak geodesic segments joining $\phi$ with $\max\{\phi-l,\psi\}, \ l >0$.
\begin{theorem} \label{ray_const}For any $\phi,\psi \in \textup{PSH}(\omega)$ with $\phi$ bounded and $\psi \leq \phi$, the weak geodesic segments $u^l$ form an increasing family. The upper semicontinuous regularization of their limit $v(\phi,\psi) = \textup{usc}(\lim_{l \to \infty} u^l)$ is a weak geodesic ray for which the following hold:
\renewcommand{\theenumi}{(\roman{enumi})}
\begin{enumerate}
\item[(i)] $v(\phi,\psi)_t = \textup{usc}(\lim_{l \to \infty} u^l_t)$ for any $t \in (0,+\infty)$.

\item[(ii)] $v(\phi,\psi)  \in \mathcal R(\phi,\psi)$, more precisely $v(\phi,\psi) = \inf_{v \in \mathcal R(\phi,\psi)}v$. In particular, $t \to v(\phi,\psi)_t$ is constant if and only if $\mathcal R(\phi,\psi)$ contains only the constant ray $\phi$.

\item[(iii)] $AM(v(\phi,\psi)_t)=AM(\phi)+c_{\psi}t$, in particular $t \to v(\phi,\psi)_t$ is constant if and only if $\psi \in \mathcal E(X,\omega)$.
\end{enumerate}
\end{theorem}
The proof will be done in a sequence of lemmas.
\begin{lemma} The weak geodesic segments $\{u^l\}_{l >0}$ form an increasing family. The upper semicontinuous regularization of their limit $v(\phi,\psi) = \textup{usc}(\lim_{l \to \infty} u^l)$ is a weak geodesic ray for which
\begin{equation}\label{usc_envelope}
v(\phi,\psi)_t = \textup{usc}(\lim_{l \to \infty} u^l_t) , \ t \in (0,+\infty).
\end{equation}
\end{lemma}
\begin{proof} It is clear that $t \to \gamma_t = \max \{ \phi-t,\psi\}, \ t >0$ is a weak subgeodesic ray. We define the following family of weak subgeodesic rays $\{t \to \gamma^l_t\}_{l \geq 0}$:
\begin{equation}\label{linearize}
\gamma^l_t =\left\{
	\begin{array}{ll}
		 u^l_t  & \mbox{if } 0 < t < l, \\
		\gamma_t & \mbox{if } t \geq l.
	\end{array}
\right.
\end{equation}
By the sub-mean value property of psh functions, it is clear that each $\gamma^l_t$ is a weak subgeodesic ray. From Berndtsson's construction it also follows that this family is increasing with $l$. In particular, the family $\{u^l\}_{l > 0}$ is also increasing in $l$. We denote $v(\phi,\psi) = \textup{usc}(\lim_{l \to \infty} u^l)$. It follows now from Bedford-Taylor theory that the Monge-Amp\`ere measures $(\omega + i\partial\bar\partial u^l)^{n+1}$ converge weakly to $(\omega + i\partial\bar\partial v(\phi,\psi))^{n+1}$.
This implies that $$(\omega + i\partial\bar\partial v(\phi,\psi))^{n+1}|_{S_{0h}\times X}=0$$ for any $h>0$.
Hence, $t \to v(\phi,\psi)_t$ is a weak geodesic ray.

We now prove (\ref{usc_envelope}). By Theorem \ref{norm_thm}, the limits $u^l_0=\lim_{t \to 0}u^l_t=\phi$ and $u^l_l = \lim_{t \to l}u^l_t=\gamma_l$ are uniform in $X$, hence it follows that
$$M_{u^l}=\sup_X \frac{u^l_l - u^l_0}{l}=\sup_X \frac{\max\{-l,\psi - \phi\}}{l}\leq 0,$$
$$m_{u^l}=\inf_X \frac{u^l_l - u^l_0}{l}=\inf_X \frac{\max\{-l,\psi - \phi\}}{l}\geq -1.$$
This implies that the $u^l$ are uniformly Lipschitz in the $t$-variable, hence so is their limit  $\lim_{l \to \infty} u^l$. This in turn implies that $(\textup{usc}(\lim_{l \to \infty} u^l))_t =  \textup{usc}(\lim_{l \to \infty} u^l_t) , \ t \in (0,+\infty),$ proving the desired result.
\end{proof}
\begin{lemma} \label{ray_const_lemma}$v(\phi,\psi)  \in \mathcal R(\phi,\psi)$, more precisely $v(\phi,\psi) = \inf_{v \in \mathcal R (\phi,\psi)}v$. In particular, $t \to v(\phi,\psi)_t$ is constant if and only if $\mathcal R(\phi,\psi)$ contains only the constant ray $\phi$.
\end{lemma}
\begin{proof} We start out by observing that $\max \{ \phi-t,\psi\}= \gamma_t \leq u^l_t \leq \phi, \ t \in (0,l), \ l >0$. After taking the limit $l \to + \infty$, then regularizing, it follows that
\begin{equation}\label{v_estimate}
\gamma_t \leq v(\phi,\psi)_t \leq \phi, \ t > 0.
\end{equation}
This implies that $\lim_{t \to +\infty}v(\phi,\psi)_t \geq \psi$ and $v(\phi,\psi)_0=\lim_{t \to 0}v(\phi,\psi)_t = \phi$. By Theorem \ref{norm_thm}, this last limit is uniform. Hence, using (\ref{v_estimate}) again and Theorem \ref{norm_thm}, we find that \begin{equation}\label{Mv}M_{v(\phi,\psi)} = \sup_X \frac{v(\phi,\psi)_t - \phi}{t} \geq \sup_X\frac{\gamma_t - \phi}{t},
\end{equation}
\begin{equation}\label{mv}
m_{v(\phi,\psi)} =\inf_X \frac{v(\phi,\psi)_t - \phi}{t}\geq \sup_X\frac{\gamma_t - \phi}{t},
\end{equation}
for any $t >0$. Since $t \to v(\phi,\psi)_t$ is decreasing it follows that $M_{v(\phi,\psi)} \leq 0$. By taking the limit $t \to + \infty$ in (\ref{Mv}) we obtain that $$M_{v(\phi,\psi)} = 0.$$
Turning to (\ref{mv}) we conclude that $$m_{v(\phi,\psi)} \geq -1.$$
To see that $v(\phi,\psi) \in \mathcal R(\phi,\psi)$ it is enough to prove that either $t \to v(\phi,\psi)_t$ is constant or $m_{v(\phi,\psi)} = -1$. To conclude this, first we prove that $v(\phi,\psi) \leq h$ for any $h \in \mathcal R(\phi,\psi)$. If $h \in \mathcal R(\phi,\psi)$, then the limit $h_0=\lim_{t \to 0} h_t =\phi$ is uniform and since $m_h=-1$ we have $h_l \geq \max\{\phi-l,\psi\}=\gamma_l, \ l >0$.  By the maximum principle this implies that $u^l_t \leq h_t, \ t \in [0,l]$. Letting $l \to +\infty$ in this estimate, then regularizing, we arrive at $v(\phi,\psi) \leq h$.

If $t \to v(\phi,\psi)_t$ is non-constant, then its normalization  ${\tilde v} \in \mathcal R(\phi,\psi)$ is non-constant as well. Since $v_0(\phi,\psi) = {\tilde v}_0=\phi$, $v(\phi,\psi) \leq {\tilde v}$  and $m_{\tilde v} = -1$,  it follows from Theorem \ref{norm_thm} that
$$m_{v(\phi,\psi)} =\inf_X \frac{v(\phi,\psi)_1 - \phi}{1} \leq \inf_X \frac{\tilde v_1 - \phi}{1} = -1.$$
This implies that $m_{v(\phi,\psi)} = -1$, finishing the proof.
\end{proof}
\begin{lemma} $AM(v(\phi,\psi)_t)=AM(\phi)+c_{\psi}t$, in particular $v(\phi,\psi)$ is constant if and only if $\psi \in \mathcal E(X,\omega)$.
\end{lemma}
\begin{proof} Since the segments $u^l$ are weak geodesics, by Theorem \ref{AM_geod} and Theorem \ref{norm_thm} we have:
$$AM(u^l_t) = AM(\phi) + \frac{t}{l}(AM(\max\{\phi - l, \psi\})-AM(\phi)), \ t \in (0,l), l >0.$$
As $\l \to \infty$, it follows from (\ref{usc_envelope}) that the sequence $u^l_t$ increases a.e. to $v(\phi,\psi)_t, \ t >0$. Using this, Bedford-Taylor theory implies that $AM(u^l_t)\to AM(v(\phi,\psi)_t)$. By Remark \ref{c_alt_def},  the right hand side of the last identity converges to $AM(\phi) + t c_\psi$. We conclude that
$$AM(v(\phi,\psi)_t) = AM(\phi) + t c_\psi, \ t \in (0, +\infty).$$
Since $v(\phi,\psi)_t \leq \phi$, by Proposition \ref{energy_dom} it follows that $t \to v(\phi,\psi)_t$ is constant if and only if $t \to AM(v(\phi,\psi)_t)$ is constant. This last condition is equivalent to $c_\psi =0$. Hence, by Theorem \ref{E_energy} it follows that $t \to v(\phi,\psi)_t$ is constant if and only if $\psi \in \mathcal E(X,\omega)$.
\end{proof}

\begin{corollary} \label{infinitypotential}Suppose $t \to u_t$ is a normalized weak geodesic ray such that $u_0 = \phi$ and $u_\infty = \psi$. Then for the normalized ray $t \to v(\phi,\psi)_t$ it is also true that $v(\phi,\psi)_\infty=\psi$.
\end{corollary}
\begin{proof} As $t \to u_t$ is a normalized, it follows that $u \in \mathcal R(\phi,\psi)$. By Theorem \ref{ray_const}(ii) it follows that $\psi \leq v(\phi,\psi)_t \leq u_t$ for all $t \in [0,\infty)$. From this the conclusion follows.
\end{proof}
\section{The inverse Legendre transform of a weak geodesic ray and $\mathcal E(X,\omega)$}

We start this section by proving a result about the maximality of the Legendre transform of a weak geodesic ray.

\begin{proposition}\label{ray_Legendre}Given a weak geodesic ray $  (0,+\infty)\ni t \to \phi_t  \in \textup{PSH}(X,\omega)$, its Legendre transform $ \Bbb R \ni \tau \to \phi^*_\tau = \inf_{t \in (0,+\infty)}(\phi_t - t\tau)\in \textup{PSH}(X,\omega)$
satisfies
$$\phi^*_\tau = P(\phi^*_\tau+C,\phi_0), \ \tau \in \Bbb R,C>0.$$
In particular, $P_{[\phi^*_\tau]}(\phi_0)=\phi^*_\tau$.
\end{proposition}
\begin{proof}Fix $\tau \in \Bbb R$. The fact that $\phi^*_\tau \in \textup{PSH}(X,\omega)$ follows from Kiselman's minimum principle. Suppose that $\phi^*_\tau \neq -\infty$ and fix $C > 0$. Since $\phi^*_\tau \leq \phi_0$, it results that $P(\phi^*_\tau+C,\phi_0) \geq \phi^*_\tau.$ Hence we only have to prove that:
$$P(\phi^*_\tau+C,\phi_0) \leq \phi^*_\tau.$$
Let $[0,1] \ni t \to g^l_t,h_t \in \textup{PSH}(X,\omega), \ l \geq 0$ be the weak geodesic segments defined by the formulas: $$g^l_t = \phi_{tl}-tl\tau,$$
$$h_t = P(\phi^*_\tau+C,\phi_0)-Ct.$$
Then we have $h_0 \leq \phi_0=\lim_{t \to 0}g^l_t=g^{l}_0$ and $h_1 \leq \phi^*_\tau \leq g^{l}_1$ for any $l \geq 0$. Hence, by the maximum principle we have
$$h_t \leq g^{l}_t, \ t \in [0,1],l \geq 0.$$
Taking the infimum in the above estimate over $l \in [0,+\infty)$ and then taking the supremum over $t \in [0,1]$,  we obtain:
$$P(\phi^*_\tau+C,\phi_0) \leq \phi^*_\tau.$$
Letting $C \to +\infty$ we obtain the last statement of the proposition.
\end{proof}

The above proposition combined together with Theorem \ref{m_norm_thm} and Theorem \ref{RWN_main} gives the following result of independent interest:

\begin{corollary} \label{RWN_main_inverse} The construction of Section 2.4 gives rise to all geodesic rays $[0,\infty) \ni t \to v_t \in \textup{PSH}(X,\o) \cap L^\infty(X)$.  
\end{corollary}
\begin{proof} We have to argue that 
$$\psi_\tau = \inf_{t \geq 0} (v_t - t\tau).$$
is a 'maximal' test curve. First we argue that $\tau \to \psi_\tau$ is a test curve. Concavity of $\tau \to \psi_\tau(z)$ follows from the convexity of the curves $t \to v_t(z)$. Condition (ii) of Definition \eqref{testcurvedef} is a consequence of Theorem \ref{m_norm_thm}. Finally, as $v_0 = \psi_{-\infty}$, maximality of the test curve $\tau \to \psi_\tau$ follows from the previous proposition:
$$\psi_\tau = P_{[\psi_\tau]}(\psi_{-\infty}).$$
\end{proof}

We will be interested in applying the above proposition in the case when $t \to \phi_t$ is a normalized geodesic ray and $\tau =0$. In this case, we have:
$$\phi^*_0 = \inf_{t \in (0,+\infty)}\phi_t = \lim_{t \to +\infty}\phi_t =: \phi_\infty.$$
Using our construction of weak geodesic rays we can now characterize $\mathcal E(X,\omega)$ in terms of envelopes:

\begin{theorem}Suppose $\psi \in \textup{PSH}(X,\omega)$ and $\phi \in \textup{PSH}(X,\omega) \cap C(X)$. Then $\psi \in \mathcal E(X, \omega)$ if and only if
$$P_{[\psi]}(\phi)=\phi.$$
\end{theorem}
\begin{proof}Suppose $\psi \in \textup{PSH}(X,\omega)$ is such that $P_{[\psi]}(\phi)=\phi.$ There exists $D >0$ such that $\psi - D < \phi$. Let $t \to v_t(\phi,\psi-D)$ be the normalized weak geodesic ray constructed in Theorem $\ref{ray_const}$. As mentioned above, in this case we have
$$v^*_{0}= \inf_{t \in (0,+\infty)}v(\phi,\psi-D)_t=\lim_{t + \infty}v_t(\phi,\psi-D) = v_\infty.$$
We also have the following sequence of inequalities:
$$\phi=P_{[\psi]}(\phi)=P_{[\psi-D]}(\phi) \leq P_{[v^*_0]}(\phi)=v^*_0\leq \phi,$$
where we have used the fact that $\psi -D \leq v_\infty = v^*_0$ and Proposition \ref{ray_Legendre}. It follows from this that $\phi = v_\infty \leq v(\phi,\psi-D)_t \leq \phi, \ t > 0 $, hence the normalized weak ray $t \to v(\phi,\psi-D)_t$ is constant equal to $\phi$. From this, using Theorem \ref{ray_const}(iii), we conclude that $\psi \in \mathcal E(X,\omega)$.

To prove the other direction, suppose now that $\psi \in \mathcal E(X,\omega)$. Since $\psi -D \leq P_{[\psi]}(\phi) \leq \phi$, it follows from Theorem \ref{E_energy} that $P_{[\psi]}(\phi) \in \mathcal E(X,\omega)$. By Proposition \ref{rwn_lemma1} $P_{[\psi]}(\phi)$ is maximal with respect to $\phi$, hence $P_{[\psi]}(\phi) \geq \phi$ a.e. with respect to $(\omega + i \partial\bar \partial P_{[\psi]}(\phi))^n$. Now, Proposition \ref{domination} implies that $P_{[\psi]}(\phi) \geq \phi$ holds everywhere.
\end{proof}

We remark in passing, that when $\psi$ is assumed to have small unbounded locus, then one can give a proof to the above result, using the maximum principle of \cite{begz} instead of Theorem $\ref{ray_const}$.

\section{Connection with analytic test configurations}

Again, we consider the weak subgeodesic ray $t \to \gamma_t = \max\{ \phi-t,\psi\}, \ t \geq 0$
and its Legendre transform $\tau \to \gamma^*_\tau$:
$$\gamma^*_\tau = \inf_{t \in [0,+\infty)}(\gamma_t - t\tau), \tau \in \Bbb R.$$
We can easily verify that $\gamma^*_\tau$ is a test curve, as defined in Section 2.4,  which can be specifically given:
\begin{equation}\label{tc_prop}
\gamma_\tau^*=\left\{
	\begin{array}{ll}
		\phi & \tau  \in (-\infty,-1),\\
        (1+\tau) \psi - \tau \phi & \tau \in [-1,0], \\
		-\infty &\tau \in (0,\infty).
	\end{array}
\right.
\end{equation}
As pointed out to us by J. Ross and D. Witt-Nystr\"om, this test curve can be seen as a generalization of a test curve arising from deformations to the normal cone.

In turns out that the geodesic ray constructed out of this test curve using the method of $\cite{rwn1}$ is the same as the one constructed in Theorem \ref{ray_const}.

\begin{theorem}Suppose $\phi,\psi \in \textup{PSH}(\omega)$ with $\phi$ bounded and $\psi \leq \phi$. Then the weak geodesic ray $v(\phi,\psi)$  is the same as the  ray obtained from the special test curve $\tau \to \gamma^*_\tau$ using the method of \cite{rwn1}.
\end{theorem}
\begin{proof} Since $\tau \to \gamma^*_\tau$ is a test curve one can apply the method of \cite{rwn1} (see Section 2.4) to produce a weak geodesic ray
$$r_t = \textup{usc}\Big(\sup_{\tau \in \Bbb R}(P_{[\gamma^*_\tau]}(\phi) + t\tau)\Big), \ t > 0.$$
We will prove that $r \in \mathcal R(\phi,\psi)$ and $r \leq u$ for any $u \in \mathcal R(\phi,\psi)$. By Theorem \ref{ray_const}(ii) this is enough to conclude that $r = v(\phi,\psi)$. Much of the remaining argument is similar to the proof of Lemma \ref{ray_const_lemma}.

If $u \in \mathcal R(\phi,\psi)$, then $\max \{ \phi-t,\psi\} = \gamma_t \leq u_t, \ t \in [0,+\infty)$, hence also $\gamma^*_\tau \leq u^*_\tau, \ \tau \in \Bbb R$. This implies that
$$P_{[\gamma^*_\tau]}(\phi)\leq P_{[u^*_\tau]}(\phi)=u^*_\tau,$$
where the last identity follows from Proposition \ref{ray_Legendre}. Hence by the involution property of Legendre transforms we arrive at
$$r_t = \textup{usc}\Big(\sup_{\tau \in \Bbb R }(P_{[\gamma^*_\tau]}(\phi) + t \tau)\Big)\leq \textup{usc}\Big(\sup_{\tau \in \Bbb R }(u^*_\tau + t\tau)\Big) = u_t, \ t \in (0,+\infty).$$

Now we prove that $r \in \mathcal R(\phi,\psi)$. Since $\gamma_t \leq r_t \leq \phi, \ t >0$, it follows that $r_0 = \lim_{t \to 0}r_t = \phi$ and $\lim_{t \to +\infty}r_t \geq \psi$. As in the proof of Lemma \ref{ray_const_lemma}, it also follows that $M_r = 0$ and $m_r \geq -1$. We have to argue that $t \to r_t$ is either constant or normalized.

If $t \to r_t$ is non-constant then its normalization $\tilde r \in \mathcal R(\phi,\psi)$ is non-constant as well. Since $r \leq \tilde r$, $r_0=\tilde r_0 = \phi$  and $m_{\tilde r}=-1$ we obtain that $m_{r} \leq -1$. This concludes the proof.
\end{proof}

\vspace{0.1 in}
\textsc{Department of Mathematics, Purdue University, West Lafayette, IN\ 47907}
\emph{E-mail address: }\texttt{\textbf{tdarvas@math.purdue.edu}}


\begin{thebibliography}{4}
\bibitem{at} C. Arezzo, G. Tian, Infnite geodesic rays in the space of K\"ahler potentials, Ann. Sc. Norm. Sup. Pisa (5) 2 (2003), no. 4, 617-630.
\bibitem{bt}E. Bedford, B.A. Tayor, 
A new capacity for plurisubharmonic functions, Acta Math. 149 (1982),  1-40.
\bibitem{bbgz} R. Berman, S. Boucksom, V. Guedj, A. Zeriahi, A variational approach to complex Monge-Amp\`ere equations,  Publ. Math. Inst. Hautes Etudes Sci. 117 (2013), 179-245.
\bibitem{begz} S. Boucksom, P. Eyssidieux, V. Guedj and A. Zeriahi, Monge-Amp\`ere equations in big cohomology
classes, Acta Math. 205 (2010), 199-262.
\bibitem{br1} B. Berndtsson, A Brunn-Minkowski type inequality for Fano manifolds and the Bando-Mabuchi uniqueness theorem,	 Invent. Math. 200 (2015), no. 1, 149--200.
\bibitem{br2}Bo Berndtsson, Probability measures related to geodesics in the space of K\"ahler metrics, arXiv:0907.1806.
\bibitem{b1}Z. Blocki, The complex Monge-Amp\`ere equation in K\"ahler geometry, CIME Summer School in Pluripotential Theory, Cetraro, July 2011, to appear in Lecture Notes in Mathematics, \url{http://gamma.im.uj.edu.pl/~blocki/publ/ln/cetr.pdf}.
\bibitem{b2}Z. Blocki, Uniqueness and stability for the Monge-Amp\`ere equation on compact K\"ahler manifolds, Indiana University Mathematics Journal 52 (2003), 1697-1702.
\bibitem{bs} M. M. Branker, M. Stawiska, Weighted pluripotential theory on complex K\"ahler manifolds, Ann. Polon. Math. 95, 2009, no. 2, 163-177, arXiv:0801.3015.
\bibitem{c1}X. X. Chen, The space of K\"ahler metrics. J. Differential Geom. 56 (2000), no. 2, 189-234.
\bibitem{c2} X. X. Chen, Space of K\"ahler metrics III: on the lower bound of the Calabi energy and geodesic distance, Invent. Math. 175 (2009), no. 3, 453-503.
\bibitem{ct} X. X. Chen,  Y. Tang, Test configuration and geodesic rays, G\'eometrie differentielle, physique math\'ematique, math\'ematiques et soci\'et\'e. I. Ast\'erisque No. 321 (2008), 139-167.
\bibitem{dl} T. Darvas, L. Lempert, Weak geodesics in the space of K\"ahler metrics, 	 Mathematical Research Letters, 19 (2012), no. 5.
\bibitem{d2}  T. Darvas, The Mabuchi completion of the space of K\"ahler potentials, { arXiv:1401.7318}.
\bibitem{d1} S. K. Donaldson, Symmetric spaces, K\"ahler geometry and Hamiltonian dynamics, Amer. Math. Soc. Transl. Ser. 2, vol. 196, Amer. Math. Soc., Providence RI, 1999, 13-33.
\bibitem{egz} P. Eyssidieux, V. Guedj, A. Zeriahi, Singular K\"ahler-Einstein metrics, Journal of the AMS 22 (2009), 607-639.
\bibitem{g}V. Guedj (editor), Complex Monge-Amp\`ere equations and geodesics in the space of K\"ahler metrics. Lecture Notes in Mathematics, 2038. Springer, Heidelberg, 2012.
\bibitem{gz} V. Guedj, A. Zeriahi, The weighted Monge-Amp\`ere energy of quasiplurisubharmonic functions, 	J. Funct. Anal. 250 (2007), no. 2, 442-482.
\bibitem{m}T. Mabuchi, Some symplectic geometry on compact K\"ahler manifolds I, Osaka J. Math. 24, 1987, 227-252.
\bibitem{ps} D. H. Phong, J. Sturm, Test configurations for K-stability and geodesic rays, J. Symplectic Geom. Volume 5, Number 2 (2007), 221-247.
\bibitem{rwn1}J. Ross, D. Witt Nystr\"om, Analytic test configurations and geodesic rays, Journal of Symplectic Geometry Volume 12, Number 1 (2014), 125--169.
\bibitem{rwn2}J. Ross, D. Witt Nystr\"om, Envelopes of positive metrics with prescribed singularities, arXiv:1210.2220.
\bibitem{rz}
Y. A. Rubinstein, S. Zelditch, The Cauchy problem for the homogeneous Monge-Amp\`ere equation, III. Lifespan, 	 Crelle's Journal 724 (2017), 105-143.
\bibitem{s}S. Semmes, Complex Monge-Amp\`ere and symplectic manifolds, Amer. J. Math. 114 (1992), 495-550.
\end{thebibliography}
\end{document}